\documentclass[12pt]{amsart}
\usepackage[margin=1in]{geometry}
\usepackage{amssymb,amsfonts,amsmath,amsthm,enumitem}
\usepackage[hidelinks]{hyperref}
\usepackage{mathtools}
\usepackage{color}
\usepackage{verbatim}

\DeclareMathOperator{\tr}{tr}
\DeclareMathOperator{\spt}{spt}

\DeclareMathOperator{\dvg}{div}

\newcommand{\ip}[2]{\langle #1, #2 \rangle}

\newcommand{\R}{\mathbb{R}}

\newcommand{\B}{\mathcal{B}}

\theoremstyle{plain}
\newtheorem{theorem}{Theorem}[section]
\newtheorem{proposition}[theorem]{Proposition}
\newtheorem{lemma}[theorem]{Lemma}
\newtheorem{corollary}[theorem]{Corollary}
\theoremstyle{definition}
\newtheorem{definition}[theorem]{Definition}
\theoremstyle{remark}
\newtheorem{remark}[theorem]{Remark}

\allowdisplaybreaks

\begin{document}

\author{Hyun Chul Jang}
\address{Department of Mathematics, Texas State University, San Marcos, TX 78666} \email{hcjang@txstate.edu}

\author{Aidan F. Wood}
\address{Department of Mathematics, University of Connecticut, Storrs, CT 06269} \email{aidan.wood@uconn.edu}

\title[Spacetime PMT with Corners via Mollification]{A spacetime positive mass theorem with corners via mollification}
\date{}

\begin{abstract}
We prove a strict dominant energy deformation theorem for asymptotically flat initial data with corners along a hypersurface $\Sigma$. The deformation preserves a corner condition on the Bartnik data across $\Sigma$. We show that if the dominant energy condition holds on each side of $\Sigma$ and the Bartnik data satisfy this corner condition, then the exterior end satisfies $E \ge |P|$ in every dimension $n \ge 3$.
\end{abstract}

\maketitle

\section{Introduction}\label{sec:intro}

An \emph{initial data set} is a triple $(X,g,k)$ where $(X,g)$ is an $n$-dimensional Riemannian manifold, possibly with boundary, and $k$ is a symmetric $(0,2)$-tensor on $X$. For a given initial data set, define the energy and current densities
\[
        \mu = \tfrac{1}{2}\bigl(R_g - |k|_g^2 + (\tr_g k)^2\bigr), \qquad J = (\dvg_g k)^{\#} - \nabla(\tr_g k).
\]
The \emph{dominant energy condition (DEC)} holds when $\mu \geq |J|_g$. Fix throughout
\begin{equation}
    \frac{n-2}{2}<q<n-2, \qquad \alpha\in(0,1), \qquad
    0<q_0<q-\frac{n-2}{2}.
\end{equation}
An initial data set $(X,g,k)$ is \emph{asymptotically flat of type $(q,q_0,\alpha)$} if there is a compact set $K\subset X$ such that $X\setminus K$ is diffeomorphic to $\R^n\setminus B$ and, in the resulting coordinates,
\begin{equation}\label{eq:af-decay}
    g-\delta \in C^{2,\alpha}_{-q}(X\setminus K), \qquad
    k \in C^{1,\alpha}_{-1-q}(X\setminus K),
\end{equation}
and the energy and current densities satisfy
\begin{equation}\label{eq:source-decay}
    \mu, J \in C^{0,\alpha}_{-n-q_0}(X\setminus K).
\end{equation}
In particular, $\mu,|J|_g\in L^1(X\setminus K,dV_g)$. Additional regularity assumptions will be specified later as needed. The \emph{ADM energy} $E$ and \emph{linear momentum} $P = (P_1, \ldots, P_n)$ of an asymptotically flat end are defined by
\begin{align}
        E &= \lim_{r \to \infty} \frac{1}{2(n-1)\omega_{n-1}} \int_{S_r} \sum_{i,j} (g_{ij,i} - g_{ii,j})\,\nu^j\,dA, \label{eq:ADM-E}\\
        P_i &= \lim_{r \to \infty} \frac{1}{(n-1)\omega_{n-1}} \int_{S_r} \sum_{j}(k_{ij} - (\tr_g k)\,g_{ij})\,\nu^j\,dA, \label{eq:ADM-P}
\end{align}
where $\omega_{n-1}$ is the volume of the unit $(n-1)$-sphere, $S_r$ is the coordinate sphere of radius $r$, and $\nu$ and $dA$ are its Euclidean outward unit normal and area measure, respectively. The weighted H\"older and Sobolev spaces used throughout are defined in Appendix~\ref{sec:weighted-spaces}.

These definitions arise from the initial value formulation of general relativity. A spacetime is a connected, time-oriented Lorentzian manifold $\mathcal{M}^{n+1}$. From this perspective, an initial data set $(X^n,g,k)$ can be viewed as a Riemannian manifold $(X^n,g)$ isometrically embedded into $\mathcal{M}^{n+1}$, with induced second fundamental form $k$. The Einstein constraint equations identify $\mu$ and $J$ with the energy density and current density, respectively, as measured by the unit normal observer. The DEC then says that the local energy--momentum density $(\mu, J)$ is future-pointing causal. The spacetime positive mass theorem is the corresponding global statement: that the ADM energy--momentum of a complete asymptotically flat initial data set satisfying the dominant energy condition is future causal, or $E \geq |P|$. Corners of the type considered here arise naturally in quasi-local mass, fill-in, and gluing constructions, as discussed below; they also fit the classical thin-shell formalism for singular hypersurfaces \cite{Israel66}.

Given an initial data set $(X,g,k)$ with boundary component $\Sigma$ and a choice of unit normal $\nu$, we define the \emph{Bartnik boundary data} induced on $\Sigma$ by $\nu$ to be the quadruple
\begin{equation}
        \B(g,k) = \bigl(\gamma,\; H,\; \omega,\; \tau\bigr),
\end{equation}
where $\gamma = g|_{T\Sigma}$ is the induced metric, $H = \dvg_\Sigma \nu$ is the mean curvature, $\omega = k(\nu, \cdot)|_{T\Sigma}$ is the connection $1$-form, and $\tau = \tr_\gamma(k|_{T\Sigma \otimes T\Sigma})$ is the tangential trace of $k$.

Let $n \geq 3$ and let $M^n$ be a smooth manifold containing a compact region $\Omega$ with smooth boundary $\partial\Omega = \Sigma$.

\begin{definition}\label{def:idc}
An \emph{initial data set with corners along} $\Sigma$ is a pair of initial data sets $(\Omega, g_-, k_-)$ and $(M\setminus\Omega, g_+, k_+)$ such that $(M\setminus\Omega, g_+, k_+)$ is asymptotically flat and the induced metrics on $\Sigma$ agree, $g_-|_{T\Sigma} = g_+|_{T\Sigma}$. We say that it \emph{satisfies the dominant energy condition} if $\mu_\pm \ge |J_\pm|_{g_\pm}$ on $\Omega$ and on $M\setminus\Omega$, respectively.
\end{definition}

Unless stated otherwise, the interior and exterior data are smooth up to $\Sigma$, and every asymptotically flat end satisfies \eqref{eq:af-decay}. Higher regularity on the asymptotic end will be stated explicitly when needed. For instance, the exterior deformation of Section~\ref{sec:strict-exterior} is formulated in $C^{m,\alpha}_{-q}\times C^{m-1,\alpha}_{-1-q}$ for a fixed $m\ge2$, while the mollification of Section~\ref{sec:mollification} requires only $C^{2,\alpha}\times C^{1,\alpha}$ regularity up to $\Sigma$ in each region.

We write $\B_\pm = \B(g_\pm,k_\pm) = (\gamma_\pm, H_\pm, \omega_\pm, \tau_\pm)$ for the Bartnik boundary data induced on $\Sigma$, where both $\B_\pm$ are computed with respect to the same unit normal $\nu$ pointing out of $\Omega$. By Definition~\ref{def:idc} the induced metrics already satisfy $\gamma_- = \gamma_+ =: \gamma$. Throughout the paper, the corner condition we consider is
\begin{equation}\label{eq:corner-condition}
    \tau_- = \tau_+ \qquad\text{and}\qquad H_- - H_+ \ge |\omega_- - \omega_+|_{\gamma} \qquad \text{on } \Sigma.
\end{equation}

In the time-symmetric case, when $k_{\pm}=0$, Miao \cite{Miao02} proved the positive mass theorem for an asymptotically flat manifold with $R_g \geq 0$ on both sides of a hypersurface $\Sigma$ and mean curvatures satisfying the corner condition $H_- \geq H_+$. Note that this agrees with the corner condition \eqref{eq:corner-condition} in the time-symmetric case. The proof smooths the metric around $\Sigma$ using convolution in Gaussian coordinates. Following Miao's construction, in Section~\ref{sec:mollification} we perform the analogous mollification of the spacetime data. Shi--Tam \cite{ShiTam2002} proved Brown--York mass positivity using asymptotically flat scalar-flat extensions and a positive mass theorem for the resulting metric. Corner positivity also plays a role in Bray's conformal flow proof of the Riemannian Penrose inequality \cite{Bray2001}.

In the spacetime case, several results are now known. Shibuya \cite{Shibuya2018} established a Lorentzian positive mass theorem with a distributional curvature condition under a spin assumption. Alaee--Yau \cite{AlaeeYau2022} proved a positive mass theorem incorporating angular momentum and charges for certain axially symmetric, maximal corner data in dimensions three and four. Tsang \cite{Tsang22} proved the three-dimensional spacetime positive mass theorem with corners under an additional exterior relative second homology hypothesis, using the spacetime harmonic function method of Hirsch--Kazaras--Khuri \cite{HKK22} and the weaker corner condition discussed in Remark~\ref{rmk:ambient-omega}. For spin initial data sets with a crease, Kazaras--Khuri--Lin \cite{KazarasKhuriLin2025} proved the spacetime positive mass theorem in all dimensions under the DEC on each side and a broader Bartnik data matching condition involving spacetime rotations; when no spacetime rotation is present, their condition reduces to Tsang's corner condition.

Our first result is a strict dominant energy deformation theorem showing that the DEC can be promoted to a strict DEC without losing the corner condition \eqref{eq:corner-condition}. 
\begin{theorem}\label{thm:strict-deform}
Let $(\Omega, g_-, k_-)$, $(M\setminus\Omega, g_+, k_+)$ be an initial data set with corners along $\Sigma$ whose exterior end is asymptotically flat of type $(q,q_0,\alpha)$. Suppose the data satisfy the dominant energy condition and their Bartnik boundary data satisfy the corner condition \eqref{eq:corner-condition}. Then for every $\epsilon > 0$ there is an initial data set with corners $(\bar g_\pm, \bar k_\pm)$ along $\Sigma$ such that
\begin{enumerate}[label=\textup{(\roman*)},nosep]
    \item $(\Omega,\bar g_-,\bar k_-)$ and $(M\setminus\Omega, \bar g_+, \bar k_+)$ satisfy the strict dominant energy condition, i.e.,
          $\bar\mu_\pm > |\bar J_\pm|_{\bar g_\pm}$,
    \item the Bartnik boundary data of $(\bar g_\pm, \bar k_\pm)$ satisfy the corner condition \eqref{eq:corner-condition}, and
    \item $\bigl|E(\bar g_+, \bar k_+) - E(g_+, k_+)\bigr|
           + \bigl|P(\bar g_+, \bar k_+) - P(g_+, k_+)\bigr| < \epsilon$.
\end{enumerate}
Moreover $(\bar g_-,\bar k_-)$ extends $C^{2,\alpha}$ up to $\Sigma$ on $\Omega$. For every fixed $m\ge2$ for which $(g_+-\delta,k_+)\in C^{m,\alpha}_{-q}\times C^{m-1,\alpha}_{-1-q}$, the exterior deformation may be chosen in the same finite weighted class and remains asymptotically flat of type $(q,q_0,\alpha)$.
\end{theorem}

The proof of Theorem~\ref{thm:strict-deform} proceeds in two steps: first, deforming the interior region $\Omega$ to strict DEC by solving a semilinear elliptic PDE and performing a conformal change (cf.~\cite{Jaracz2025}); second, deforming on the exterior $M\setminus \Omega$ to strict DEC while controlling Bartnik data by applying a proposition of Hirsch--Huang \cite{HH2025} which uses the modified constraint operator introduced by Corvino--Huang \cite{CorvHuang2020}. 

We track the ADM energy--momentum across these deformations and apply Miao's mollification to the strict DEC data, which smooths the data but possibly loses the DEC in a small neighborhood of $\Sigma$. Correcting for this region by solving a linear PDE and performing another conformal change, we obtain a single initial data set satisfying the DEC everywhere, and arrive at the following spacetime positive mass theorem with corners. 

\begin{theorem}\label{thm:main}
Let $(\Omega, g_-, k_-)$, $(M\setminus\Omega, g_+, k_+)$ be an initial data set with corners along $\Sigma$ whose exterior end is asymptotically flat of type $(q,q_0,\alpha)$.\footnote{When $n\ge8$, we further assume that $g_+-\delta\in C^{m,\alpha}_{-q}$ and $k_+\in C^{m-1,\alpha}_{-1-q}$ for every integer $m\ge2$, as required in our application of \cite{BrendleWang}.} Suppose the data satisfy the dominant energy condition and their Bartnik boundary data satisfy the corner condition \eqref{eq:corner-condition}. Then the ADM energy--momentum $(E, P)$ of the exterior end $(M\setminus\Omega, g_+, k_+)$ satisfies
\[
    E \ge |P|.
\]
\end{theorem}

The proof reduces, via the aforementioned strict-DEC deformation, mollification, and conformal correction, to an application of known spacetime positive mass results to data arbitrarily close to $(g_+,k_+)$ in ADM energy--momentum. For $3\le n\le7$, we apply Eichmair--Huang--Lee--Schoen \cite[Theorem~1]{EHLS} directly. For $n\ge8$, under the additional asymptotic regularity specified in the theorem, the reduction following \cite[Theorem~18]{EHLS} brings the problem into the setting of Brendle--Wang \cite[Theorem~1.1]{BrendleWang}; see Step~4 of the proof of Theorem~\ref{thm:main}.

Theorem~\ref{thm:main} can also be read as a statement solely about the exterior, in the sense that an asymptotically flat initial data set with boundary has $E \ge |P|$ if its boundary data admit a dominant energy fill-in.

\begin{corollary}\label{cor:fill-in}
    Let $(M \setminus \Omega, g_+, k_+)$ be an asymptotically flat initial data set of type $(q,q_0,\alpha)$ with boundary $\Sigma = \partial\Omega$ satisfying the dominant energy condition, with the additional weighted regularity required in Theorem~\ref{thm:main} when $n\ge8$, and let $\B_+ = \B(g_+, k_+)$ be its Bartnik boundary data on $\Sigma$. Suppose there exists a compact initial data set $(\Omega, g_-, k_-)$ satisfying the dominant energy condition, inducing the metric $\gamma_+$ on $\Sigma$, and whose Bartnik boundary data $\B_-$, together with $\B_+$, satisfy the corner condition \eqref{eq:corner-condition}. Then the ADM energy--momentum of $(M \setminus \Omega, g_+, k_+)$ satisfies $E \ge |P|$.
\end{corollary}

Corollary~\ref{cor:fill-in} can be interpreted as an obstruction to fill-ins. An asymptotically flat dominant energy initial data set with boundary and $E < |P|$ admits no dominant energy fill-in of its boundary data compatible with \eqref{eq:corner-condition}. In the time-symmetric case this is the obstruction to filling in the boundary data of a negative-mass extension by a metric of nonnegative scalar curvature, in the spirit of the quasi-local mass positivity results of Shi--Tam \cite{ShiTam2002}.

In Section~\ref{sec:strict}, we deform to strict DEC while preserving the corner inequality across $\Sigma$, and prove Theorem~\ref{thm:strict-deform}. Section~\ref{sec:mollification} applies Miao's mollification to the deformed data. Section~\ref{sec:conformal-pde} solves the linear conformal PDE and verifies the DEC. We prove Theorem~\ref{thm:main} in Section~\ref{sec:main-proof}.

\medskip

\noindent\textbf{Acknowledgments.} The authors are grateful to Lan-Hsuan Huang for helpful discussions and suggestions.

\section{Strict DEC deformation}\label{sec:strict}
From this point on we work on the glued manifold $M=\Omega\cup\Sigma\cup(M\setminus\Omega)$ and fix a smooth function $r:M\to[1,\infty)$ that coincides with the radial coordinate $|x|$ of the asymptotic chart outside the compact set $K$, so that expressions such as $(1+r)^{-a}$ are defined on all of $M$. Throughout the paper, $C$ denotes a positive constant depending only on $n$, $(\Omega, g_-, k_-)$, or $(M\setminus \Omega, g_+, k_+)$, and the same letter $C$ may be used for different constants.
\subsection{Conformal deformation on the compact region}\label{sec:strict-interior}

We begin by deforming the interior region $(\Omega, g_-, k_-)$ to achieve the strict DEC. Throughout this section we fix $p>n/(1-\alpha)$, so that $p>n$ and $\alpha<1-n/p$. The following proposition applies to any smooth compact initial data set. 

\begin{proposition}\label{prop:strict-interior}
Let $(\Omega, g, k)$ be a smooth compact initial data set with boundary $\Sigma$ satisfying $\mu \ge |J|_g$. For any $\lambda > 0$, there exists a conformal deformation $\bar{g} = u^{4/(n-2)}g$, $\bar{k} = u^{2/(n-2)}k$ such that:
\begin{enumerate}[label=\textup{(\alph*)}]
    \item $u \ge 1$ on $\Omega$ and $u = 1$ on $\Sigma$,
    \item $\bar\mu - |\bar J|_{\bar g} \ge \lambda\, u^{-(n+2)/(n-2)} > 0$ on $\Omega$,
    \item $\bar\gamma = \gamma$, $\bar\omega = \omega$, and $\bar\tau = \tau$ on $\Sigma$,
    \item $\|u - 1\|_{C^{1,\alpha}(\Omega)} \le C\lambda$ for a constant $C$ depending only on $(\Omega, g, k)$,
    \item $u \in C^{2,\alpha}(\Omega)$.
\end{enumerate}
Setting
\begin{equation} \label{eq:delta-H-def}
    \delta_H := -c_n^{-1}\,\partial_\nu u\big|_\Sigma,
\end{equation}
we have $\delta_H \ge 0$ on $\Sigma$, $\bar H_- = H_- - \delta_H$, and $\|\delta_H\|_{C^{1,\alpha}(\Sigma)} \le C\lambda$. Here $c_n = \frac{n-2}{2(n-1)}$.
\end{proposition}

We prove Proposition~\ref{prop:strict-interior} at the end of this subsection. Setting $\kappa = |k|_g$ and $\lambda' = c_n\lambda > 0$, we consider the equation
\begin{equation} \label{eq:nonlinear-pde}
        \Delta_g u + \kappa|\nabla u|_g = -\lambda', \qquad u = 1 \text{ on } \Sigma.
\end{equation}
If $u \ge 1$ solves \eqref{eq:nonlinear-pde}, then $\Delta_g u = -\kappa|\nabla u|_g - \lambda'$, and substituting into the conformal inequality \eqref{eq:conf-ineq} for the dominant energy scalar found in Appendix~\ref{sec:appendix} yields
\begin{align}
    \bar\mu - |\bar J|_{\bar g}
    &\ge u^{-4/(n-2)}\Bigl(\mu - |J|_g + c_n^{-1}\,u^{-1}\kappa|\nabla u|_g + c_n^{-1}\,u^{-1}\lambda' - c_n^{-1}\,u^{-1}\kappa|\nabla u|_g\Bigr) \notag\\
    &= u^{-4/(n-2)}\bigl(\mu - |J|_g + \lambda\,u^{-1}\bigr) \notag\\
    &\ge \lambda\,u^{-(n+2)/(n-2)} > 0,
\end{align}
using $\mu \ge |J|_g$ and $u \ge 1$. It therefore suffices to establish existence of a solution $u \ge 1$ to \eqref{eq:nonlinear-pde} with the stated estimates.

\begin{remark}
Equation \eqref{eq:nonlinear-pde} is a semilinear elliptic equation whose nonlinearity is Lipschitz in $\nabla u$, with linear growth. Structurally similar equations appear in \cite{HKK22}, where the coefficient of the gradient term is $\tr_gk$, rather than $|k|_g$. The Leray--Schauder construction below is also closely related to \cite[Section~3.1]{Jaracz2025}, which treats a three-dimensional Neumann problem.
\end{remark}

We devote the remainder of this subsection to the existence theory for \eqref{eq:nonlinear-pde}.

\begin{lemma}[Uniqueness]\label{lem:uniqueness}
The equation \eqref{eq:nonlinear-pde} has at most one solution $u \in W^{2,p}(\Omega) \cap C^0(\Omega)$ for $p > n$.
\end{lemma}

\begin{proof}
Let $u_1, u_2$ be two solutions and set $\eta = u_1 - u_2$. Where $|\nabla u_1|_g + |\nabla u_2|_g \neq 0$, set
\begin{equation} \label{eq:V-def}
    V := \frac{\nabla u_1 + \nabla u_2}{|\nabla u_1|_g + |\nabla u_2|_g},
\end{equation}
and $V := 0$ otherwise, so that $|V|_g \le 1$ a.e. The difference-of-norms factorization then gives $|\nabla u_1|_g - |\nabla u_2|_g = \ip{V}{\nabla\eta}_g$ a.e.\ on $\Omega$. Subtracting the two equations yields
\[
    \Delta_g\eta + \kappa\ip{V}{\nabla\eta}_g = 0 \quad \text{in } \Omega, \qquad \eta = 0 \text{ on } \Sigma.
\]
This is a linear equation for $\eta$ with first-order coefficient $b := \kappa V \in L^\infty(\Omega)$ and zeroth-order coefficient $c = 0$. Since $c \le 0$, the weak maximum principle gives
\[
    \sup_\Omega \eta \le \sup_\Sigma \eta^+ = 0, \qquad \sup_\Omega (-\eta) \le \sup_\Sigma (-\eta)^+ = 0,
\]
so $\eta \equiv 0$.
\end{proof}

\begin{lemma}[A priori bounds]\label{lem:apriori}
Any solution $u_\sigma$ to the homotopy
\begin{equation} \label{eq:homotopy}
    \Delta_g u_\sigma + \sigma\kappa|\nabla u_\sigma|_g = -\sigma\lambda', \qquad u_\sigma = 1 \text{ on } \Sigma, \qquad \sigma \in [0,1],
\end{equation}
satisfies $u_\sigma \ge 1$ and $\|u_\sigma - 1\|_{C^{1,\alpha}(\Omega)} \le C\lambda'$ for a constant $C$ independent of $\sigma$.
\end{lemma}

\begin{proof}
Since $\Delta_g u_\sigma=-\sigma\kappa|\nabla u_\sigma|_g-\sigma\lambda'\le0$, $u_\sigma$ is superharmonic, so $\min_{\Omega}u_\sigma=\min_\Sigma u_\sigma=1$, which establishes the lower bound. Now we show the $L^\infty$ bound uniform in $\sigma$. Set $w_\sigma=u_\sigma-1\ge0$, so $w_\sigma=0$ on $\Sigma$ and
\[
    \Delta_gw_\sigma+b_\sigma\cdot\nabla w_\sigma = -\sigma\lambda', \qquad
    b_\sigma:=\sigma\kappa\,\frac{\nabla u_\sigma}{|\nabla u_\sigma|_g}\in L^\infty(\Omega), \quad
    |b_\sigma|_g\le\kappa \text{ for every } \sigma\in[0,1],
\]
where $b_\sigma$ is set to $0$ where $\nabla u_\sigma=0$, as in \eqref{eq:V-def}. By the usual maximum principle estimates, since $w_\sigma = 0$ on $\Sigma$,
\[
    \sup_\Omega w_\sigma \le C_1\,\|\sigma\lambda'\|_{L^n(\Omega)} \le C_1\lambda',
\]
where $C_1=C_1(n,(\Omega,g),\sup_\Omega\kappa)$ depends only on the fixed bound $\|b_\sigma\|_{L^\infty}\le\sup_\Omega\kappa$ (not on $\sigma$ itself), and $\sup_\Omega(-w_\sigma)\le0$ since $w_\sigma\ge0$. Hence
\begin{equation}\label{eq:winfty-uniform}
    \|w_\sigma\|_{L^\infty(\Omega)}\le C_1\lambda',
\end{equation}
with $C_1$ independent of $\sigma$.

By the global $W^{2,p}$ estimate \cite[Theorem~9.13]{GT}, we have
\[
    \|w_\sigma\|_{W^{2,p}(\Omega)} \le C_2\bigl(\|\sigma\lambda'\|_{L^p(\Omega)} + \|w_\sigma\|_{L^p(\Omega)}\bigr),
\]
where $C_2=C_2(n,p,(\Omega,g),\sup_\Omega\kappa)$ depends only on the fixed bound $\|b_\sigma\|_{L^\infty}\le\sup_\Omega\kappa$, hence is independent of $\sigma$. Using $\|w_\sigma\|_{L^p(\Omega)}\le\mathrm{Vol}(\Omega)^{1/p}\|w_\sigma\|_{L^\infty(\Omega)}$ and \eqref{eq:winfty-uniform},
\[
    \|w_\sigma\|_{W^{2,p}(\Omega)} \le C_2\lambda' + C_2\,\mathrm{Vol}(\Omega)^{1/p}C_1\lambda' = C_3\lambda',
\]
with $C_3$ independent of $\sigma$. The Morrey embedding $W^{2,p}(\Omega)\hookrightarrow C^{1,\alpha}(\Omega)$ for $p>n$ yields $\|w_\sigma\|_{C^{1,\alpha}(\Omega)}\le C\lambda'$, independent of $\sigma$.
\end{proof}

\begin{corollary}\label{cor:existence-compact}
For any $\lambda' > 0$, the equation \eqref{eq:nonlinear-pde} admits a unique solution $u \in W^{2,p}(\Omega) \cap C^{1,\alpha}(\Omega)$ with $u \ge 1$. Moreover, $\|u - 1\|_{C^{1,\alpha}(\Omega)} \le C\lambda'$.
\end{corollary}

\begin{proof}
Set $v = u - 1$, so that \eqref{eq:nonlinear-pde} becomes
\[
    \Delta_g v + \kappa|\nabla v|_g = -\lambda', \qquad v = 0 \text{ on } \Sigma.
\]
Let $B = C^{1,\alpha}_0(\Omega) := \{w \in C^{1,\alpha}(\Omega) : w = 0 \text{ on } \Sigma\}$. For $\tilde v \in B$ and $\sigma \in [0,1]$, let $w = T(\tilde v, \sigma)$ be the unique $W^{2,p} \cap W^{1,p}_0$ solution of the Dirichlet problem
\begin{equation}
    \Delta_g w = -\sigma\bigl(\lambda' + \kappa|\nabla\tilde v|_g\bigr), \qquad w = 0 \text{ on } \Sigma.
\end{equation}
Existence and uniqueness of this solution are guaranteed by \cite[Theorem~9.15]{GT}: the source lies in $L^\infty(\Omega) \subset L^p(\Omega)$, and the operator is $\Delta_g$, with no lower-order terms.

The map $\tilde v \mapsto \kappa|\nabla\tilde v|_g$ is continuous from $B$ to $C^0(\Omega) \subset L^p(\Omega)$, since $\bigl||\nabla\tilde v_1|_g - |\nabla\tilde v_2|_g\bigr| \le |\nabla\tilde v_1 - \nabla\tilde v_2|_g$ pointwise. Composing with the solution operator $(\Delta_g)^{-1}\colon L^p(\Omega) \to W^{2,p}(\Omega)\cap W^{1,p}_0(\Omega)$ of the Dirichlet problem and the compact embedding $W^{2,p}(\Omega) \hookrightarrow\hookrightarrow C^{1,\alpha}(\Omega)$ for $p > n$ shows that $T : B \times [0,1] \to B$ is continuous and maps bounded sets to precompact sets, so is itself compact.

For $\sigma = 0$ the source vanishes, thus $T(\cdot, 0) \equiv 0$ and $v_0 = 0$ is its unique fixed point. A fixed point $v = T(v, \sigma)$ solves
\[
    \Delta_g v + \sigma\kappa|\nabla v|_g = -\sigma\lambda', \qquad v = 0 \text{ on } \Sigma,
\]
which is the homotopy \eqref{eq:homotopy}. Therefore Lemma~\ref{lem:apriori} gives $u_\sigma \ge 1$ and $\|v\|_{C^{1,\alpha}} = \|u_\sigma - 1\|_{C^{1,\alpha}} \le C\lambda'$ for $u_\sigma = v+1$, a bound uniform in $\sigma$. By the Leray--Schauder fixed point theorem \cite[Theorem~11.6]{GT}, $T(\cdot, 1)$ has a fixed point $v$, which solves $\Delta_g v + \kappa|\nabla v|_g = -\lambda'$. Hence $u = v + 1 \ge 1$ solves \eqref{eq:nonlinear-pde}, and uniqueness follows from Lemma~\ref{lem:uniqueness}. The estimate $\|u - 1\|_{C^{1,\alpha}} \le C\lambda'$ is precisely the a priori bound from Lemma~\ref{lem:apriori}.
\end{proof}

\begin{proof}[Proof of Proposition~\ref{prop:strict-interior}]
Apply Corollary~\ref{cor:existence-compact} with $\lambda' = c_n\lambda$. Properties (a)--(c) follow from $u = 1$ on $\Sigma$ and the discussion above. Property (d) follows from Corollary~\ref{cor:existence-compact}, which gives $\|u-1\|_{C^{1,\alpha}} \le C\lambda' = Cc_n\lambda$. The mean curvature change \eqref{eq:H-change} together with \eqref{eq:delta-H-def} gives $\bar H_- = H_- - \delta_H$. Finally, $\delta_H \ge 0$ on $\Sigma$, since $u \ge 1$ achieves its minimum value $1$ on $\Sigma$ and $\nu$ points outward, so $\partial_\nu u \le 0$.

It remains to prove (e) and the stated boundary estimate. From Corollary~\ref{cor:existence-compact}, $u\in C^{1,\alpha}(\Omega)$, so $|\nabla u|_g\in C^{0,\alpha}(\Omega)$. Since the original compact data are smooth up through the boundary, $\kappa=|k|_g\in C^{0,\alpha}(\Omega)$ and
\[
    F:=-\lambda'-\kappa|\nabla u|_g\in C^{0,\alpha}(\Omega),
    \qquad
    \|F\|_{C^{0,\alpha}(\Omega)}\le C\lambda.
\]
Equation \eqref{eq:nonlinear-pde} reads $\Delta_gu=F$, with zero Dirichlet data for $u-1$. Standard boundary Schauder estimates therefore give
\[
    \|u-1\|_{C^{2,\alpha}(\Omega)}
       \le C\bigl(\|F\|_{C^{0,\alpha}(\Omega)}+\|u-1\|_{C^0(\Omega)}\bigr)
       \le C\lambda.
\]
This proves (e). Taking the normal derivative and its trace on $\Sigma$ yields $\|\delta_H\|_{C^{1,\alpha}(\Sigma)}\le C\lambda$.
\end{proof}

\subsection{Matching deformation on the exterior}\label{sec:strict-exterior}

The exterior deformation produces a strict dominant energy gap with a definite decay rate. Fix
\begin{equation}\label{eq:q1-choice}
    q_1\in\Bigl(\max(q,n-2+q_0),\;q+\tfrac{n-2}{2}\Bigr].
\end{equation}
The gap below decays like $(1+r)^{-2-q_1}$ and is tracked at this rate through the remaining deformations.

\begin{proposition}\label{prop:exterior-matching}
Let $(M \setminus\Omega, g, k)$ be an asymptotically flat initial data set of type $(q,q_0,\alpha)$ with boundary $\Sigma$ satisfying $\mu \ge |J|_g$, and assume $(g-\delta,\,k)\in C^{m,\alpha}_{-q}(M\setminus\Omega)\times C^{m-1,\alpha}_{-1-q}(M\setminus\Omega)$ for some fixed $m\ge2$, where $\delta$ is the Euclidean metric in the asymptotic chart. Let $\B^{m,\alpha}(\Sigma)$ denote the space of Bartnik boundary data with $\gamma\in C^{m,\alpha}(\Sigma)$ and $H,\omega,\tau\in C^{m-1,\alpha}(\Sigma)$, and let $\mathcal{D}$ be Bartnik boundary data on $\Sigma$ satisfying $\|\mathcal{D} - \B(g, k)\|_{\B^{m,\alpha}(\Sigma)} \le \lambda$. Then for $\lambda$ sufficiently small, there exists a deformation $(\bar{g}, \bar{k})$ on $M \setminus \Omega$ such that:
\begin{enumerate}[label=\textup{(\alph*)}]
    \item $\B(\bar{g}, \bar{k}) = \mathcal{D}$ on $\Sigma$,
    \item $\bar\mu - |\bar J|_{\bar g} \ge \lambda(1+r)^{-2-q_1}$,
    \item $|E(\bar{g}, \bar{k}) - E(g, k)| + |P(\bar{g}, \bar{k}) - P(g, k)| \to 0$
          as $\lambda\to0$,
    \item $\bar\mu,\bar J\in C^{0,\alpha}_{-n-q_0}(M\setminus\Omega)$.
\end{enumerate}
The deformation belongs to $C^{m,\alpha}_{-q}\times C^{m-1,\alpha}_{-1-q}$ relative to the Euclidean background on the end.
\end{proposition}

\begin{proof}
For this subsection, write
\[
    j:=J^\flat_g=\dvg_g k-d(\tr_gk)
\]
for the momentum one-form of the background data, and for variable data $(\tilde g,\tilde k)$ write $\tilde j:=(J(\tilde g,\tilde k))^{\flat_{\tilde g}}$. Thus $|j|_g=|J|_g$. We use the one-form convention for the constraint map $\Phi(\tilde g,\tilde k)=(\mu(\tilde g,\tilde k),\tilde j)$ and the modified constraint operator $\overline{\Phi}_{(g,k)}$ introduced by Corvino--Huang \cite{CorvHuang2020}, with the normalization of Hirsch--Huang \cite[Definition~A.2]{HH2025} with $\varphi = 0$. In particular,
\[
    \overline{\Phi}_{(g,k)}(h,w) = \Phi(h,w) + \bigl(0,\; \tfrac12\,h \cdot j\bigr),
\]
where $j$ is fixed and $(h\cdot j)_i=g^{ab}h_{ia}j_b$. The operator $\overline{\Phi}_{(g,k)}$ is defined on a neighborhood of $(g,k)$ in $C^{m,\alpha}_{-q}(M \setminus \Omega) \times C^{m-1,\alpha}_{-1-q}(M \setminus \Omega)$ and takes values in $C^{m-2,\alpha}_{-2-q}(M \setminus \Omega)$.

By \cite[Proposition~5.3(2)]{HH2025}, the combined map
\[
    T(\tilde g, \tilde k) = \bigl(\overline{\Phi}_{(g,k)}(\tilde g, \tilde k),\;
    \B(\tilde g, \tilde k)\bigr)
\]
is locally surjective at $(g, k)$. Set $\rho := (1+r)^{-2-q_1}$ on $M\setminus\Omega$. For $\lambda>0$ small, local surjectivity produces $(\bar g,\bar k)$ satisfying
\begin{equation} \label{eq:ext-deform}
    \overline{\Phi}_{(g,k)}(\bar{g}, \bar{k})
    = \overline{\Phi}_{(g,k)}(g, k) + \lambda(\rho, 0),
    \qquad \B(\bar{g}, \bar{k}) = \mathcal{D},
\end{equation}
with perturbation size $O(\lambda)$ in $C^{m,\alpha}_{-q}\times C^{m-1,\alpha}_{-1-q}$; in particular $|\bar g-g|_g<1$ for $\lambda$ small. By \cite[Lemma~A.4]{HH2025}, applied with $\varphi=0$ and $W=0$,
\[
    \bar\mu - |\bar j|_{\bar g}
    \ge \mu-|j|_g+\lambda\rho
    \ge \lambda\rho
    =\lambda(1+r)^{-2-q_1},
\]
which proves (b), since $|\bar j|_{\bar g}=|\bar J|_{\bar g}$.

The two components of \eqref{eq:ext-deform} give
\[
    \bar\mu-\mu=\lambda\rho,
    \qquad
    \bar j-j=-\tfrac12(\bar g-g)\cdot j.
\]
Because $\rho\in C^{0,\alpha}_{-n-q_0}$, $j\in C^{0,\alpha}_{-n-q_0}$, and $\bar g-g\in C^{m,\alpha}_{-q}$, we have $\bar j\in C^{0,\alpha}_{-n-q_0}$. Raising the index with $\bar g$ gives the same decay for $\bar J$, proving (d). These formulas also show that the difference $(\bar\mu-\mu,\bar j-j)$ tends to zero in the weighted $L^1$ space required in \cite[Proposition~19]{EHLS}. Together with $(\bar g,\bar k)\to(g,k)$ in $C^{m,\alpha}_{-q}\times C^{m-1,\alpha}_{-1-q}$, \cite[Proposition~19]{EHLS} yields
\[
    \bigl|E(\bar g,\bar k)-E(g,k)\bigr|
    +\bigl|P(\bar g,\bar k)-P(g,k)\bigr|\longrightarrow0,
\]
which proves (c).

The $C^{m,\alpha}_{-q}\times C^{m-1,\alpha}_{-1-q}$ regularity of the deformation is part of the local-surjectivity conclusion in \cite[Proposition~5.3]{HH2025}.
\end{proof}

\subsection{Proof of Theorem~\ref{thm:strict-deform}}\label{sec:deform-pf}

\begin{proof}[Proof of Theorem~\ref{thm:strict-deform}]
Let $\epsilon > 0$ be given and $\lambda > 0$ a deformation parameter, not yet fixed.

First we deform the interior $(\Omega, g_-, k_-)$. Proposition~\ref{prop:strict-interior} applied to $(\Omega, g_-, k_-)$ produces a conformal deformation $\bar g_- = u^{4/(n-2)} g_-$, $\bar k_- = u^{2/(n-2)} k_-$ satisfying the strict dominant energy condition, whose Bartnik data differ from $\B(g_-, k_-)$ only in the mean curvature: $\bar H_- = H_- - \delta_H$, with $\bar\gamma_- = \gamma_-$, $\bar\omega_- = \omega_-$, and $\bar\tau_- = \tau_-$. By Proposition~\ref{prop:strict-interior}(e), $u \in C^{2,\alpha}(\Omega)$, so $\bar g_-$ and $\bar k_-$ are $C^{2,\alpha}$ on $\Omega$.

The function $\delta_H$ (the conformally induced change in mean curvature) is $C^{1,\alpha}(\Sigma)$ with $\|\delta_H\|_{C^{1,\alpha}(\Sigma)}=O(\lambda)$. This is sufficient for the target Bartnik data when $m=2$, but not for the arbitrary higher finite regularity allowed in Proposition~\ref{prop:exterior-matching}. We therefore use the following crude estimate which has the desired sign in the corner inequality. Since $\delta_H \ge 0$ on $\Sigma$, the constant
\begin{equation}
    \delta_{H,\max} := \max_\Sigma \delta_H = \|\delta_H\|_{C^0(\Sigma)}
\end{equation}
satisfies $\delta_{H,\max} \ge \delta_H$ pointwise on $\Sigma$, $\delta_{H,\max} \in C^\infty(\Sigma)$, and $\|\delta_{H,\max}\|_{C^{m,\alpha}(\Sigma)}=O(\lambda)$ for every $m\ge0$.

Next we deform the exterior $(M\setminus \Omega, g_+, k_+)$. Fix an integer $m\ge2$ for which $(g_+-\delta,k_+)\in C^{m,\alpha}_{-q}\times C^{m-1,\alpha}_{-1-q}$; in the default case, take $m=2$. Set
\begin{equation}
    \mathcal{D} := \bigl(\gamma_+,\; H_+ - \delta_{H,\max},\; \omega_+,\; \tau_+\bigr),
\end{equation}
which satisfies
\[
    \|\mathcal{D} - \B(g_+, k_+)\|_{\B^{m,\alpha}(\Sigma)} \le C_0\lambda
\]
for a fixed constant $C_0>0$. Set $\lambda_+:=C_0\lambda$. For $\lambda$ small, Proposition~\ref{prop:exterior-matching}, with $\lambda_+$ in place of its deformation parameter, produces exterior data $(\bar g_+,\bar k_+)$ satisfying
\[
    \B(\bar g_+,\bar k_+)=\mathcal D,
    \qquad
    \bar\mu_+-|\bar J_+|_{\bar g_+}
       \ge \lambda_+(1+r)^{-2-q_1}>0.
\]
For the fixed regularity order, the same proposition gives $(\bar g_+-\delta,\bar k_+)\in C^{m,\alpha}_{-q}\times C^{m-1,\alpha}_{-1-q}$.

We can then verify the conclusions. For (i), Proposition~\ref{prop:strict-interior}(b) gives
\[
    \bar\mu_- - |\bar J_-|_{\bar g_-}
    \ge \lambda u^{-(n+2)/(n-2)}>0
    \qquad\text{on }\Omega,
\]
and the exterior bound above holds on $M\setminus\Omega$. Thus both regions satisfy the strict dominant energy condition. For (ii), the induced metrics and tangential traces match, $\bar\gamma_-=\bar\gamma_+$ and $\bar\tau_-=\bar\tau_+$, and
\begin{equation}
\begin{aligned}
    \bar H_- - \bar H_+
      &= (H_- - H_+) + (\delta_{H,\max} - \delta_H)\\
      &\ge |\omega_- - \omega_+|_\gamma
       = |\bar\omega_- - \bar\omega_+|_\gamma.
\end{aligned}
\end{equation}
Here we used the original corner inequality and $\delta_{H,\max}-\delta_H\ge0$ on $\Sigma$. Thus the deformed Bartnik data satisfy \eqref{eq:corner-condition}. Finally, Proposition~\ref{prop:exterior-matching}(c) gives
\[
    |E(\bar g_+,\bar k_+)-E(g_+,k_+)|
    +|P(\bar g_+,\bar k_+)-P(g_+,k_+)|\longrightarrow0
\]
as $\lambda\to0$. Fixing $\lambda$ sufficiently small proves (iii).
\end{proof}

\begin{remark}
    The interior deformation of Theorem~\ref{thm:strict-deform} is closely related to density theorems for the DEC. In the smooth setting, Schoen--Yau promoted the dominant energy condition to its strict form in $n=3$ by conformal change \cite[Lemma~1]{SY81} (see the remark following \cite[Theorem~22]{EHLS} concerning an error in the proof of that lemma). Eichmair--Huang--Lee--Schoen did this in all dimensions $n \ge 3$ \cite[Theorem~18]{EHLS}, deforming the full data through the modified constraint operator, with mass control as in \cite[Proposition~19]{EHLS}. This improvability of the dominant energy scalar was studied further by Huang--Lee \cite{HL24}. The correction term that absorbs the first-order change of $|J|$ under such a deformation appears already at the linearized level in \cite[Lemma~20]{EHLS}, where it is chosen precisely so that the first-order variation of $|J|_{g}$ vanishes.

    Our interior step recovers the conformal route of \cite{SY81} and extends it to all dimensions $n \ge 3$ and to data with corners. Instead of linearizing the map $\varphi \mapsto |\bar J|_{\bar g}$, which may fail to be differentiable on $\{J=0\}$, we carry the gradient nonlinearity in equation~\eqref{eq:nonlinear-pde} itself and solve that equation directly.
\end{remark}

\section{Mollification of initial data with corners}\label{sec:mollification}

\begin{proposition}\label{prop:miao}
Let $(\Omega, \bar g_-, \bar k_-)$, $(M\setminus\Omega, \bar g_+, \bar k_+)$ be an initial data set with corners along $\Sigma$ satisfying the dominant energy condition on each side, whose Bartnik boundary data satisfy the corner condition \eqref{eq:corner-condition}. Suppose $\bar g_\pm$ extend $C^{2,\alpha}$ and $\bar k_\pm$ extend $C^{1,\alpha}$ up to $\Sigma$ on $\Omega$ and $\overline{M\setminus\Omega}$, respectively. For each $\delta > 0$ small, there exists an initial data set $(g_\delta, k_\delta)$ of class $C^{2,\alpha}\times C^{1,\alpha}$ on $M$ such that:
\begin{enumerate}[label=\textup{(\alph*)}]
    \item $(g_\delta, k_\delta)$ is obtained by mollifying
      $(\bar{g}_\pm, \bar{k}_\pm)$ in Gaussian normal coordinates within the
      $\delta$-tubular neighborhood $U_\delta$ of $\Sigma$,
    \item $(g_\delta, k_\delta) = (\bar{g}_\pm, \bar{k}_\pm)$ outside $U_\delta$,
    \item $(g_\delta,k_\delta)$ is asymptotically flat of type $(q,q_0,\alpha)$ and has
          exactly the same asymptotic regularity and decay of the energy and current densities as
          $(\bar g_+,\bar k_+)$,
    \item the dominant energy scalar satisfies
    \begin{equation} \label{eq:dec-scalar}
            (\mu_\delta - |J_\delta|_{g_\delta})_- = O(1) \quad \text{uniformly in } \delta,
    \end{equation}
    with $\spt\bigl((\mu_\delta - |J_\delta|_{g_\delta})_-\bigr) \subset U_\delta$ and
    \begin{equation} \label{eq:Ln2-est}
            \bigl\|(\mu_\delta - |J_\delta|_{g_\delta})_-\bigr\|_{L^{n/2}(M)} \le C\delta^{2/n} \to 0 \quad \text{as } \delta \to 0,
    \end{equation}
    \item $E(g_\delta, k_\delta) = E(\bar{g}_+, \bar{k}_+)$ and $P(g_\delta, k_\delta) = P(\bar{g}_+, \bar{k}_+)$ for $\delta>0$ small.
\end{enumerate}
\end{proposition}
 
\begin{proof}
We adapt the smoothing argument of Miao~\cite{Miao02} to initial data sets, performing the convolution in the transverse direction only while handling both $\bar g$ and $\bar k$. Fix a tubular neighborhood $U = \Sigma \times (-2\varepsilon, 2\varepsilon)$ of $\Sigma$ in which Gaussian normal coordinates with respect to $\Sigma$ are defined. In these coordinates the ambient data take the form
\[
    \bar g = \bar g_{\alpha\beta}(x,t)\,dx^\alpha dx^\beta + dt^2, \qquad \bar g_{\alpha\beta}(x,t) = \begin{cases} (\bar g_-)_{\alpha\beta}(x,t), & t < 0, \\ (\bar g_+)_{\alpha\beta}(x,t), & t \ge 0,\end{cases}
\]
with an analogous split $\bar k_{ij} = (\bar k_-)_{ij}$ for $t<0$ and $\bar k_{ij} = (\bar k_+)_{ij}$ for $t\ge 0$. The components $\bar g_{\alpha\beta}$ are $C^{2,\alpha}$ and the components $\bar k_{ij}$ are $C^{1,\alpha}$ up to $\Sigma$ from each side by hypothesis, and bounded on $\overline U$. The tangential metric components agree across $\Sigma$ because $\bar\gamma_-=\bar\gamma_+$, while only the tangential traces of $\bar k_-$ and $\bar k_+$ are required to agree. Thus the transverse derivatives of the metric, as well as $\bar k_{tt}$, $\bar k_{t\alpha}$ and the trace-free part of $\bar k_{\alpha\beta}$, may jump at $t=0$.
 
\medskip \noindent\textit{Step 1: Construction of the mollified data.} Let $\phi \in C^\infty_c((-1,1))$ be a nonnegative even function with $\int_\R\phi = 1$. Choose an even function $\sigma\in C^\infty_c((-1,1))$ such that $0\le\sigma\le1/100$ and $\sigma\equiv1/100$ on $[-1/2,1/2]$, and set
\[
    \sigma_\delta(t):=\delta^2\sigma(t/\delta).
\]
Then $\sigma_\delta=\delta^2/100$ for $|t|\le\delta/2$, $\sigma_\delta=0$ for $|t|\ge\delta$, and
\begin{equation}\label{eq:sigma-derivative-bounds}
    |\sigma_\delta'|\le C\delta,
    \qquad
    |\sigma_\delta''|\le C,
\end{equation}
with $C$ independent of $\delta$.

Define the smoothed tangential metric and smoothed symmetric tensor on $U_\delta := \Sigma \times (-\delta, \delta)$ by convolution in $t$ at the variable scale $\sigma_\delta(t)$:
\begin{equation}\label{eq:mollified-def}
    (g_\delta)_{\alpha\beta}(x,t) := \int_\R \bar g_{\alpha\beta}(x, t - \sigma_\delta(t) s)\,\phi(s)\,ds, \qquad (k_\delta)_{ij}(x,t) := \int_\R \bar k_{ij}(x, t - \sigma_\delta(t) s)\,\phi(s)\,ds,
\end{equation}
while keeping $(g_\delta)_{tt} \equiv 1$ and $(g_\delta)_{t\alpha} \equiv 0$. Extend by $(g_\delta, k_\delta) := (\bar g_\pm, \bar k_\pm)$ on $M \setminus U_\delta$. Since $\sigma$ is compactly supported in $(-1,1)$, the scale $\sigma_\delta$ vanishes on a neighborhood of $\{|t|=\delta\}$, so this extension has class $C^{2,\alpha}\times C^{1,\alpha}$. Properties (a)--(c) of the proposition follow directly.
 
Now set
\[
    \Psi_\delta(t) := \tfrac{100}{\delta^2}\,\phi\!\left(\tfrac{100\,t}{\delta^2}\right).
\]
This is a nonnegative approximate Dirac delta supported in $[-\delta^2/100,\delta^2/100]$, and $\int_\R\Psi_\delta=1$.

On $\spt\Psi_\delta$ we have $|t|\le\delta^2/100$ and $\sigma_\delta(t)=\delta^2/100$. Since the one-sided metrics are $C^1$ up to $\Sigma$ and have the same tangential value $\gamma$ there,
\begin{equation}\label{eq:core-metric-estimate}
    |g_\delta(x,t)-\gamma(x)|_\gamma
    +|g_\delta^{-1}(x,t)-\gamma^{-1}(x)|_\gamma
    \le C\delta^2
    \qquad\text{on }\spt\Psi_\delta.
\end{equation}
Because $\delta^2\Psi_\delta=O(1)$, replacing a $g_\delta$-contraction or norm by its $\gamma$ counterpart in a coefficient multiplying $\Psi_\delta$ changes the result only by $O(1)$.
 
\medskip \noindent\textit{Step 2: Energy density estimate.} Because $(g_\delta)_{tt}\equiv 1$ and $(g_\delta)_{t\alpha}\equiv 0$, the vector $\nu = \partial_t$ is the unit normal to every level set $\Sigma_t := \Sigma\times\{t\}$ for the whole family $g_\delta$. The foliation identity (see \cite[Equation~(3)]{Miao02})
\[
    R_{g_\delta} = R_{\Sigma_t} - \bigl(|A_\delta|^2_{g_\delta} + H_\delta^2\bigr) - 2\partial_t H_\delta
\]
expresses the ambient scalar curvature in terms of the intrinsic scalar curvature of $\Sigma_t$, its second fundamental form $A_\delta = \tfrac12\partial_t g_\delta|_{T\Sigma_t}$ (with sign convention matching $H=\dvg_\Sigma\nu$), mean curvature $H_\delta = g_\delta^{\alpha\beta}(A_\delta)_{\alpha\beta}$, and the transverse derivative $\partial_t H_\delta$. The first three terms involve only zeroth- and first-order transverse derivatives of $g_\delta$ together with tangential derivatives of $(g_\delta)_{\alpha\beta}$ of order at most two. Since convolution in $t$ commutes with tangential derivatives and $\bar g_\pm$ is $C^{2,\alpha}$ up to $\Sigma$ from each side, these contributions are uniformly bounded in $\delta$ by \cite[Lemmas~3.2, 3.3]{Miao02}. The singular behavior is localized in $\partial_t H_\delta$. On $|t|<\delta/2$, the scale is constant, and \cite[Equation~(30)]{Miao02} gives
\[
    \partial_t H_\delta(x,t) = O(1) + \bigl(\bar H_+(x) - \bar H_-(x)\bigr)\,\Psi_\delta(t)
\]
throughout that strip. On $\delta/2\le|t|\le\delta$, the convolution samples only one side of the corner. Differentiating twice under the integral gives terms involving $\partial_t^2\bar g\,(1-\sigma_\delta's)^2$ and $-\partial_t\bar g\,\sigma_\delta''s$, which are uniformly bounded by \eqref{eq:sigma-derivative-bounds}. Hence the same estimate holds on all of $U_\delta$, with $\Psi_\delta=0$ on the outer strip. Since $k_\delta$ is a convolution of a uniformly bounded tensor, $(\tr_{g_\delta}k_\delta)^2 - |k_\delta|_{g_\delta}^2$ is $O(1)$. Combining,
\begin{equation}\label{eq:mu-delta-est}
    \mu_\delta = O(1) + (\bar H_- - \bar H_+)\,\Psi_\delta(t).
\end{equation}
 
\medskip \noindent\textit{Step 3: Current density estimate.} We show
\begin{equation}\label{eq:J-delta-est}
    |J_\delta|_{g_\delta} \le |\bar\omega_- - \bar\omega_+|_{\gamma}\,\Psi_\delta(t) + O(1).
\end{equation}
Write $j_\delta:=(J_\delta)^{\flat_{g_\delta}}=\dvg_{g_\delta}k_\delta-d(\tr_{g_\delta}k_\delta)$. At a fixed $x\in\Sigma$, choose a $\gamma$-orthonormal tangential frame $(e_\alpha)_{\alpha=1}^{n-1}$ and extend it $t$-independently using the product identification. Since the Christoffel symbols for $g_\delta$ are uniformly bounded (first derivatives of the mollification of a Lipschitz metric) and $k_\delta$ itself is bounded, the contribution of the covariant-derivative correction $\Gamma\cdot k_\delta$ to $j_\delta$ is $O(1)$. Moreover, tangential derivatives $\partial_{x^\alpha}(k_\delta)_{ij}$ equal convolutions of $\partial_{x^\alpha}\bar k_{ij}$ and are therefore $O(1)$. The only potentially singular terms come from $\partial_t(k_\delta)_{ij}$.
 
On the strip $|t|<\delta/2$ we have $\sigma_\delta \equiv \delta^2/100$, so \eqref{eq:mollified-def} reduces to the standard $t$-convolution $k_\delta = \bar k *_t \Psi_\delta$. Splitting the integral at $\tau=0$ and integrating by parts, using $\bar k \in L^\infty$ with one-sided limits $\bar k^\pm$ at $\tau=0$ and bounded transverse derivatives on each side, we get
\begin{align}
    \partial_t(k_\delta)_{ij}(x,t)
        &= -\int_\R \bar k_{ij}(x,\tau)\,\partial_\tau\Psi_\delta(t-\tau)\,d\tau \notag\\
        &= \bigl(\bar k_{ij}^+(x) - \bar k_{ij}^-(x)\bigr)\,\Psi_\delta(t) + O(1). \label{eq:dt-kdelta-jump}
\end{align}

On the outer strip $\delta/2\le|t|\le\delta$ (say $t>0$; $t<0$ is symmetric), the argument $t-\sigma_\delta(t)s$ appearing in \eqref{eq:mollified-def} satisfies, for $s\in\spt(\phi)\subset[-1,1]$,
\[
    t-\sigma_\delta(t)s \in [t-\delta^2/100,\ t+\delta^2/100] \subset (\delta/2-\delta^2/100,\ 2\delta]
    \subset (0,2\delta]
\]
for $\delta$ small, so this range never meets $t=0$. Consequently $\bar k_{ij}(x,\cdot) = \bar k_{ij}^+(x,\cdot)$ is a single $C^{1,\alpha}$ function of $t$ throughout this range, and differentiating \eqref{eq:mollified-def} under the integral sign gives
\[
    \partial_t(k_\delta)_{ij}(x,t) = \int_\R \partial_\tau\bar k_{ij}^+(x,\tau)\Big|_{\tau=t-\sigma_\delta(t)s}
    \bigl(1-\sigma_\delta'(t)s\bigr)\,\phi(s)\,ds.
\]
Since $\bar k^+_{ij}$ has bounded transverse derivative uniformly on $(0,2\delta]$ and $|\sigma_\delta'(t)|\le C\delta$, $|s|\le1$ on $\spt(\phi)$, the factor $|1-\sigma_\delta'(t)s|\le1+C\delta=O(1)$, giving directly
\[
    \partial_t(k_\delta)_{ij}(x,t) = O(1) \quad\text{on } \delta/2\le|t|\le\delta,
\]
consistent with $\Psi_\delta\equiv0$ there (since $|100t/\delta^2|\ge50/\delta>1$, so the argument lies outside $\spt(\phi)$). By differentiating under the integral before any integration by parts, we avoid generating a $\sigma_\delta'$-related singular term, as no jump lies within the bounds of integration. Thus \eqref{eq:dt-kdelta-jump} holds on all of $U_\delta$.
 
We now assemble the components of $j_\delta$ in the frame $(\partial_t,e_\alpha)$. Using $(g_\delta)_{tt}\equiv 1$, $(g_\delta)_{t\alpha}\equiv 0$, the uniform boundedness of $\Gamma_{g_\delta}$, and the boundedness of $\partial_t g_\delta^{\alpha\beta}$, a direct computation yields
\[
    (j_\delta)_t = -\partial_t\bigl(\tr_{\Sigma_t}k_\delta\bigr) + O(1), \qquad
    (j_\delta)_\alpha = \partial_t(k_\delta)_{t\alpha} + O(1),
\]
where $\tr_{\Sigma_t}k_\delta := g_\delta^{\alpha\beta}(k_\delta)_{\alpha\beta}$. Applying \eqref{eq:dt-kdelta-jump}, the boundedness of $\partial_tg_\delta^{\alpha\beta}$, and \eqref{eq:core-metric-estimate}, we obtain
\[
    \partial_t(\tr_{\Sigma_t}k_\delta)
       =(\bar\tau_+-\bar\tau_-)\Psi_\delta+O(1).
\]
Thus the $\tau$-matching $\bar\tau_- = \bar\tau_+$ gives $(j_\delta)_t=O(1)$. Writing $c_\alpha := \bar k_{t\alpha}^+ - \bar k_{t\alpha}^-$, so that $\vec c := (c_\alpha)_\alpha$ satisfies $|\vec c|_\gamma = |\bar\omega_--\bar\omega_+|_\gamma$ (since $\bar\omega_\pm(e_\alpha) = \bar k_\pm(\partial_t,e_\alpha) = \bar k_{t\alpha}^\pm$ in this frame), the triangle inequality and \eqref{eq:core-metric-estimate} give
\[
    |(j_\delta)^\top|_{g_\delta}
    \le |\vec c|_{g_\delta}\,\Psi_\delta(t)+O(1)
    = |\bar\omega_--\bar\omega_+|_\gamma\,\Psi_\delta(t)+O(1).
\]
Since $(j_\delta)_t=O(1)$ and $|j_\delta|_{g_\delta}=|J_\delta|_{g_\delta}$, this yields \eqref{eq:J-delta-est}.

\medskip \noindent\textit{Step 4: DEC estimate and $L^{n/2}$ bound.} Combining \eqref{eq:mu-delta-est} and \eqref{eq:J-delta-est}, there is a constant $C>0$, independent of $\delta$, such that
\begin{equation}
    \mu_\delta - |J_\delta|_{g_\delta} \ge -C + \zeta\,\Psi_\delta(t), \qquad
    \zeta := (\bar H_- - \bar H_+) - |\bar\omega_- - \bar\omega_+|_{\gamma}.
\end{equation}
By the corner inequality, $\zeta \ge 0$ pointwise on $\Sigma$, so the singular term is nonnegative. Consequently,
\[
    \bigl(\mu_\delta - |J_\delta|_{g_\delta}\bigr)_- \le C,
\]
which is bounded in $L^\infty(M)$ uniformly in $\delta$. Outside $U_\delta$ the data are unchanged and the original data satisfy the DEC on each side, so the negative part is supported in $U_\delta$. Since $|U_\delta| \le C\delta$, H\"older's inequality gives
\[
    \bigl\|(\mu_\delta - |J_\delta|_{g_\delta})_-\bigr\|_{L^{n/2}(M)} \le \bigl\|(\mu_\delta - |J_\delta|_{g_\delta})_-\bigr\|_{L^\infty}\cdot|U_\delta|^{2/n} \le C'\delta^{2/n},
\]
establishing \eqref{eq:dec-scalar}--\eqref{eq:Ln2-est}.
 
\medskip \noindent\textit{Step 5: ADM equality.} The data $(g_\delta, k_\delta)$ agree with $(\bar g_+, \bar k_+)$ on $(M\setminus\Omega)\setminus U_\delta$, and the ADM integrals \eqref{eq:ADM-E}--\eqref{eq:ADM-P} are limits of integrals over coordinate spheres $S_r$ at infinity. For every $\delta$ small, $U_\delta$ is contained in a fixed compact set disjoint from the asymptotic region, so every coordinate sphere $S_r$ used in \eqref{eq:ADM-E}--\eqref{eq:ADM-P} lies outside $U_\delta$, and the integrands for $(g_\delta,k_\delta)$ and $(\bar g_+,\bar k_+)$ coincide identically on each $S_r$. Hence $E(g_\delta, k_\delta) = E(\bar g_+, \bar k_+)$ and $P(g_\delta, k_\delta) = P(\bar g_+, \bar k_+)$.
\end{proof}

\begin{remark}\label{rmk:ambient-omega}
An alternative formulation of the corner condition \eqref{eq:corner-condition}, used by Tsang in dimension three under the spacetime harmonic function approach \cite{Tsang22}, combines the $\tau$-matching and the tangential connection-$1$-form inequality into a single scalar inequality. Define the \emph{ambient connection $1$-form} $\omega^{\mathrm{amb}}(X) := \pi(\nu, X)$ on all of $TM|_\Sigma$ (where $\pi = k - (\tr_g k)g$). Then
\[
    |\omega^{\mathrm{amb}}_- - \omega^{\mathrm{amb}}_+|_g^2 = (\tau_- - \tau_+)^2 + |\omega_- - \omega_+|_\gamma^2,
\]
so the assumption that $\bar\tau_- = \bar\tau_+,\ \bar H_- - \bar H_+ \ge |\bar\omega_- - \bar\omega_+|_\gamma$ is strictly stronger than the single inequality $\bar H_- - \bar H_+ \ge |\bar\omega^{\mathrm{amb}}_- - \bar\omega^{\mathrm{amb}}_+|_{\bar g}$. But the mollification computation goes through under the weaker hypothesis, since in Step~3 the $\tau$-matching is invoked only to discard $(\bar\tau_- - \bar\tau_+)\Psi_\delta$ from $(J_\delta)_t$ before taking the norm, and without it one obtains $|J_\delta|_{g_\delta} \le |\bar\omega^{\mathrm{amb}}_- - \bar\omega^{\mathrm{amb}}_+|_{\bar g}\,\Psi_\delta + O(1)$ directly. We retain the Bartnik-component formulation in Theorem~\ref{thm:main} because it matches the structure of the deformation in Section~\ref{sec:strict}, which preserves $\bar\gamma$, $\bar\omega$, and $\bar\tau$ on $\Sigma$ and shifts only $\bar H$.
\end{remark}
 
\begin{remark}
The mollification estimate \eqref{eq:Ln2-est} depends on $\zeta$ only through its sign. As in Miao's general Riemannian corner setting $\bar H_-\ge\bar H_+$ \cite{Miao02}, a strict inequality produces a nonnegative delta contribution, while in the exactly matching case the coefficient vanishes. Here the spacetime coefficient is $\zeta=(\bar H_- - \bar H_+)-|\bar\omega_- - \bar\omega_+|_\gamma$. It may be positive, in which case $\zeta\Psi_\delta$ is unbounded as $\delta\to0$, but the estimate is unaffected because we control the negative part $(\mu_\delta-|J_\delta|_{g_\delta})_-$ rather than the full dominant energy scalar. Thus any nonnegative singular contribution drops out, and only the bounded $O(1)$ remainder enters the $L^{n/2}$ bound.
\end{remark}

When Proposition~\ref{prop:miao} is applied to the data produced in Section~\ref{sec:strict}, its DEC hypothesis holds strictly, and property~(b) leaves those data unchanged outside $U_\delta$. Consequently, the key estimates on $(g_\delta,k_\delta)$ outside $U_\delta$ are the following, where $a_0,c_0>0$ are the strict DEC gap constants of Step~1 in Section~\ref{sec:main-proof}:
\begin{itemize}[nosep]
    \item on $\Omega \setminus U_\delta$: $\mu_\delta - |J_\delta|_{g_\delta} = \bar\mu_- - |\bar J_-|_{\bar g_-} \ge a_0 > 0$,
    \item on $(M \setminus \Omega) \setminus U_\delta$: $\mu_\delta - |J_\delta|_{g_\delta} = \bar\mu_+ - |\bar J_+|_{\bar g_+} \ge c_0(1+r)^{-2-q_1}$.
\end{itemize}
These strict DEC gaps, inherited from the deformation in Section~\ref{sec:strict}, will be used to absorb the gradient term in the conformal correction in the next section.

\section{Conformal correction: regularity and the dominant energy condition}\label{sec:conformal-pde}

\subsection{General conformal DEC-restoration}\label{sec:generic-dec-restore}

Fix throughout this section
\begin{equation}\label{eq:p-choice}
    p > \frac{n}{1-\alpha}.
\end{equation}
By \eqref{eq:q1-choice}, $2+q_1>n+q_0$, so $(1+r)^{-2-q_1}\in C^{0,\alpha}_{-n-q_0}\cap L^1$. Moreover, $q_1<2q$ because $q_1\le q+(n-2)/2$ and $q>(n-2)/2$. Throughout this section, a \emph{$C^{2,\alpha}$ asymptotically flat initial data set of type $(q,q_0,\alpha)$} is a complete initial data set $(N,h,K)$ without boundary, with $h\in C^{2,\alpha}_{\mathrm{loc}}(N)$ and $K\in C^{1,\alpha}_{\mathrm{loc}}(N)$, whose end satisfies \eqref{eq:af-decay} and \eqref{eq:source-decay}.

The following proposition solves any linear conformal equation whose source $f$ only needs to dominate the negative part of the dominant energy scalar pointwise. 

\begin{proposition}[Conformal DEC restoration]\label{prop:generic-dec}
Let $(N,h,K)$ be a $C^{2,\alpha}$ asymptotically flat initial data set of type $(q,q_0,\alpha)$, let $W\Subset W'\Subset N$ be open sets with compact closure, and let $c_*,\Lambda_1>0$. Then there is an $\varepsilon_0=\varepsilon_0\bigl(n,p,q,q_0,q_1,\alpha,c_*,\Lambda_1,(N,h,K)\bigr)>0$ such that the following holds. Suppose $f\in C^{0,\alpha}(N)$ satisfies
\begin{enumerate}[label=\textup{(F\arabic*)}]
    \item $f\ge0$ on $N$ and $\spt(f)\subset\overline{W'}$;
    \item $\mu_h-|J_h|_h \ge -f + 1$ on $W$, and
          $\mu_h-|J_h|_h \ge c_*\,(1+r)^{-2-q_1}$ on $N\setminus W$;
    \item $\|f\|_{L^\infty(N)}\le\Lambda_1$ and
          $\|f\|_{L^{n/2}(N)} < \varepsilon_0$.
\end{enumerate}
Then
\begin{equation}\label{eq:generic-linear-pde}
    -c_n^{-1}\Delta_h u = f\,u, \qquad u\to1 \text{ at infinity},
\end{equation}
admits a unique positive solution $u\in C^{2,\alpha}_{\mathrm{loc}}(N)$. It satisfies $u\ge1$ and admits the expansion
\[
    u=1+A\,r^{2-n}+\omega, \qquad
    \omega=O_{2,\alpha}(r^{2-n-\eta})
\]
for some $\eta>0$. Moreover:
\begin{enumerate}[label=\textup{(\alph*)}]
    \item $u-1\in W^{2,p}_{-q}(N)$, with
          $\|u-1\|_{W^{2,p}_{-q}(N)}\le C\|f\|_{L^p(N)}$;
          in particular $1\le u\le1+C\|f\|_{L^p}$ and
          $|\nabla u|_h\le C\|f\|_{L^p}(1+r)^{-1-q}$;
    \item for $\widetilde h:=u^{4/(n-2)}h$ and
          $\widetilde K:=u^{2/(n-2)}K$, there is a constant $c'>0$ such that
          \[
              \widetilde\mu-|\widetilde J|_{\widetilde h}
              \ge c'(1+r)^{-2-q_1}\quad\text{on }N;
          \]
    \item $(n-2)\,\omega_{n-1}\,A = c_n\displaystyle\int_N f\,u\,dV_h$, so
          $|A|\le C\|f\|_{L^\infty(N)}\,\mathrm{Vol}_h(\overline{W'})$;
    \item $E(\widetilde h,\widetilde K)=E(h,K)+2A$ and
          $P(\widetilde h,\widetilde K)=P(h,K)$;
    \item $(\widetilde h,\widetilde K)$ is asymptotically flat of type
          $(q,q_0,\alpha)$. If $(h,K)$ is smooth outside a compact set, then so are
          $(\widetilde h,\widetilde K)$.
\end{enumerate}
\end{proposition}

\begin{proof}
First note that
\[
\|f\|_{L^p(N)}
\le\|f\|_{L^\infty}^{1-n/(2p)}\|f\|_{L^{n/2}}^{n/(2p)}
\le\Lambda_1^{1-n/(2p)}\varepsilon_0^{n/(2p)},
\]
so every smallness requirement on $\|f\|_{L^p}$ below follows from taking $\varepsilon_0$ small.

\medskip

\textit{Existence and asymptotics.} Write \eqref{eq:generic-linear-pde} as $\Delta_hu+au=0$, $a:=c_nf\ge0$. By (F3) and \cite[Lemma~4.1]{Miao02} (due to Schoen--Yau \cite{SY79}, and the argument holds for all $n\ge3$), for $\varepsilon_0$ small enough there is a positive $C^2$ solution tending to $1$ at infinity. The same lemma gives
\[
    u=1+Ar^{2-n}+\omega,
    \qquad
    \omega=O(r^{1-n}),\quad \partial\omega=O(r^{-n}).
\]
Local Schauder estimates upgrade $u$ to $C^{2,\alpha}_{\mathrm{loc}}$. Since $f$ has compact support, $u$ is $h$-harmonic near infinity. Applying the exterior Schauder estimate \cite[Theorem~1]{Meyers63} to the equation satisfied by $\omega=u-1-Ar^{2-n}$ gives
\[
    \omega=O_{2,\alpha}(r^{2-n-\eta})
\]
for some $\eta>0$. In particular, $\nabla u=O_{1,\alpha}(r^{1-n})$.

\medskip

\textit{Lower bound.} Since $\Delta_hu=-au\le0$ and $u\to1$ at infinity, the minimum principle gives $u\ge1$.

\medskip

\textit{Uniqueness.} If $u_1,u_2$ are positive solutions of \eqref{eq:generic-linear-pde} with $u_i\to1$, set $v=u_1-u_2$. Then $\Delta_hv=-av$, and $v$ is harmonic near infinity with $v\to0$. The expansion just established gives $v=O(r^{2-n})$ and $|\nabla v|_h=O(r^{1-n})$. In particular, $v=O(r^{-q})$ and $|\nabla v|_h=O(r^{-1-q})$ since $q<n-2$. Exhausting $N$ by coordinate balls $B_R$, the boundary term satisfies
\[
    \left|\int_{\partial B_R}v\,\partial_\nu v\,dA_h\right|
    \le CR^{n-2-2q}\longrightarrow0,
\]
because $q>(n-2)/2$. Hence
\[
    \int_N|\nabla v|_h^2\,dV_h
    =\int_Nav^2\,dV_h
    \le C\|f\|_{L^{n/2}}
          \|v\|_{L^{2n/(n-2)}}^2
    \le C\|f\|_{L^{n/2}}\int_N|\nabla v|_h^2\,dV_h.
\]
For $\varepsilon_0$ small, $v\equiv0$.

\medskip

\textit{Part (a).} Set $v:=u-1\ge0$. Then $\Delta_hv=-a(v+1)=:G$, with $\spt(G)\subset\overline{W'}$. If $m_0:=\|v\|_{L^\infty}$, then $\|G\|_{L^p}\le c_n(1+m_0)\|f\|_{L^p}$. The expansion above and $q<n-2$ show that $v\in W^{2,p}_{-q}$. Since $2-n<-q<0$, Bartnik's weighted Laplacian isomorphism \cite[Proposition~2.2]{Bartnik86} and $\spt(G)\subset\overline{W'}$ give
\[
    \|v\|_{W^{2,p}_{-q}}
    \le C\|G\|_{L^p_{-q-2}}
    \le C'(1+m_0)\|f\|_{L^p}.
\]
The weighted Morrey embedding $W^{2,p}_{-q}\hookrightarrow C^{1,\alpha}_{-q}$ then gives
\[
    m_0\le C''(1+m_0)\|f\|_{L^p}.
\]
After decreasing $\varepsilon_0$, the last inequality can be absorbed, yielding the asserted estimates.

\medskip

\textit{Part (b).} By \eqref{eq:conf-ineq} and \eqref{eq:generic-linear-pde},
\[
    \widetilde\mu-|\widetilde J|_{\widetilde h}
    \ge u^{-4/(n-2)}\Bigl((\mu_h-|J_h|_h)+f-\mathcal E\Bigr),
    \qquad
    \mathcal E:=c_n^{-1}u^{-1}|K|_h|\nabla u|_h.
\]
On $W$, (F2) and part (a) give $(\mu_h-|J_h|_h)+f-\mathcal E\ge1-C\|f\|_{L^p}>0$. On $N\setminus W$,
\[
    \mathcal E\le C\|f\|_{L^p}(1+r)^{-2-2q}
    \le C\|f\|_{L^p}(1+r)^{-2-q_1},
\]
where the last inequality follows from $q_1<2q$. Thus
\[
    (\mu_h-|J_h|_h)+f-\mathcal E
    \ge(c_*-C\|f\|_{L^p})(1+r)^{-2-q_1}
\]
on $N\setminus W$. Since $W$ is precompact and $u$ is bounded above, the two region estimates combine, after decreasing $\varepsilon_0$, to give the asserted global constant $c'>0$.

\medskip

\textit{Parts (c)--(d).} The expansion of $u$ and the computation in \cite[Equations~(8)--(9)]{EHLS} give $E(\widetilde h,\widetilde K)=E(h,K)+2A$, while the decay of $u-1$ shows directly from \eqref{eq:ADM-P} that $P(\widetilde h,\widetilde K)=P(h,K)$. Fix $R$ with $\overline{W'}\subset B_R$. Integrating \eqref{eq:generic-linear-pde} over $B_{R'}$, $R'>R$, and then letting $R'\to\infty$ gives
\[
    (n-2)\omega_{n-1}A=c_n\int_Nfu\,dV_h.
\]
The stated bound for $A$ follows from the upper bound for $u$ in part (a).

\medskip

\textit{Part (e).} Since $q<n-2$ and $u-1=Ar^{2-n}+O_{2,\alpha}(r^{2-n-\eta})$, the conformally transformed metric and tensor satisfy
\[
    \widetilde h-\delta\in C^{2,\alpha}_{-q},
    \qquad
    \widetilde K\in C^{1,\alpha}_{-1-q}.
\]
On the complement of $\overline{W'}$, where $\Delta_hu=0$, the conformal formulas give
\[
    \widetilde\mu=u^{-4/(n-2)}\mu_h,
    \qquad
    \widetilde J=u^{-6/(n-2)}
       \Bigl(J_h+c_n^{-1}u^{-1}K^\#(\nabla u)\Bigr).
\]
Here $K^\#(\nabla u)=O_{0,\alpha}(r^{-n-q})$, because $K=O_{1,\alpha}(r^{-1-q})$ and $\nabla u=O_{1,\alpha}(r^{1-n})$. Since $q>q_0$, this term belongs to $C^{0,\alpha}_{-n-q_0}$. Together with \eqref{eq:source-decay}, this proves the required decay of the energy and current densities for $(\widetilde h,\widetilde K)$. Finally, if $(h,K)$ is smooth outside a compact set, then $u$ is smooth there by elliptic bootstrapping, because $f$ is compactly supported, and thus the transformed data are also smooth outside a compact set.
\end{proof}

\subsection{Restoring the dominant energy condition}\label{sec:application-dec}

Here we apply Proposition~\ref{prop:generic-dec} to the mollified data. Recall from Section~\ref{sec:mollification} that the mollified data satisfy $(\mu_\delta-|J_\delta|_{g_\delta})_-\le C$ with support in $U_\delta$ and $\|(\mu_\delta-|J_\delta|_{g_\delta})_-\|_{L^{n/2}(M)}\le C\delta^{2/n}$. To force positivity on $U_\delta$, we dominate the negative part of the dominant energy scalar with a smooth function. Let $\chi_\delta\in C^\infty_c(U_{2\delta})$ with $\chi_\delta\equiv1$ on $U_\delta$, $0\le\chi_\delta\le1$, and set
\begin{equation}\label{eq:f-delta-smooth}
    f_\delta := \Lambda_\delta\,\chi_\delta, \qquad
    \Lambda_\delta := \bigl\|(\mu_\delta-|J_\delta|_{g_\delta})_-\bigr\|_{L^\infty(M)} + 1,
\end{equation}
so that $f_\delta\in C^\infty(M)$, $f_\delta\ge(\mu_\delta-|J_\delta|_{g_\delta})_-+1$ on $U_\delta$, $f_\delta\ge(\mu_\delta-|J_\delta|_{g_\delta})_-$ pointwise and $\spt(f_\delta)\subset U_{2\delta}$. Moreover, $\Lambda_\delta\le\Lambda$ uniformly by \eqref{eq:dec-scalar}, and hence
\[
    \|f_\delta\|_{L^{n/2}(M)}
    \le \Lambda\,|U_{2\delta}|^{2/n}
    \le C\delta^{2/n}\longrightarrow0.
\]

\begin{corollary}\label{cor:app-dec}
For $\delta$ sufficiently small, Proposition~\ref{prop:generic-dec} applies with $(N,h,K)=(M,g_\delta,k_\delta)$, $f=f_\delta$, $W=U_\delta$, $W'=U_{2\delta}$, $c_*=\min(a_0,c_0)$, and $\Lambda_1=\Lambda$, yielding $u_\delta$ and
\[
    (\widetilde g_\delta,\widetilde k_\delta) := \bigl(u_\delta^{4/(n-2)}g_\delta,\; u_\delta^{2/(n-2)}k_\delta\bigr).
\]
The resulting data have class $C^{2,\alpha}\times C^{1,\alpha}$, are asymptotically flat of type $(q,q_0,\alpha)$, and satisfy
\[
    \widetilde\mu_\delta-|\widetilde J_\delta|_{\widetilde g_\delta}
    \ge c_\delta(1+r)^{-2-q_1}
\]
for some $c_\delta>0$. Moreover,
\[
    E(\widetilde g_\delta,\widetilde k_\delta)=E(\bar g_+,\bar k_+)+O(\delta),
    \qquad P(\widetilde g_\delta,\widetilde k_\delta)=P(\bar g_+,\bar k_+).
\]
\end{corollary}

\begin{proof}
All that remains is to check the conditions needed to apply Proposition~\ref{prop:generic-dec}. Condition (F1) holds by construction of $f_\delta$. For (F2), we check the two regions: on $W=U_\delta$, \eqref{eq:f-delta-smooth} gives $f_\delta\ge(\mu_\delta-|J_\delta|_{g_\delta})_-+1$, so $\mu_\delta-|J_\delta|_{g_\delta}\ge-(\mu_\delta-|J_\delta|_{g_\delta})_-\ge-f_\delta+1$. On $M\setminus U_\delta$, the strict DEC estimates inherited from Section~\ref{sec:strict} give $\mu_\delta-|J_\delta|_{g_\delta}=\bar\mu_--|\bar J_-|_{\bar g_-}\ge a_0$ on $\Omega\setminus U_\delta$ and $\mu_\delta-|J_\delta|_{g_\delta}=\bar\mu_+-|\bar J_+|_{\bar g_+}\ge c_0(1+r)^{-2-q_1}$ on $(M\setminus\Omega)\setminus U_\delta$. Therefore $\mu_\delta-|J_\delta|_{g_\delta}\ge c_*(1+r)^{-2-q_1}$, where $c_*=\min(a_0,c_0)$, using $(1+r)^{-2-q_1}\le1$ on the compact region $\Omega\setminus U_\delta$.

Finally, for (F3), we have that $\|f_\delta\|_{L^\infty}=\Lambda_\delta\le\Lambda_1$ uniformly in $\delta$ by \eqref{eq:dec-scalar}, and $\|f_\delta\|_{L^{n/2}}\le C\delta^{2/n}\to0$, while $\varepsilon_0$ may be taken independent of $\delta$: the family $(g_\delta,k_\delta)$ agrees with $(\bar g_\pm,\bar k_\pm)$ outside a fixed compact set, is uniformly equivalent to a fixed smooth metric, and has uniformly bounded first derivatives. Thus the Sobolev and elliptic constants used in Proposition~\ref{prop:generic-dec}, as well as $\sup_M|k_\delta|$, can be chosen independently of $\delta$. Proposition~\ref{prop:miao} gives $E(g_\delta,k_\delta)=E(\bar g_+,\bar k_+)$, $P(g_\delta,k_\delta)=P(\bar g_+,\bar k_+)$ exactly. 

Combining with Proposition~\ref{prop:generic-dec}(c)--(d), $\mathrm{Vol}(U_{2\delta})=O(\delta)$, and $\|f_\delta\|_{L^\infty}=O(1)$ gives $A_\delta=O(\delta)$, hence the desired mass estimate. The remaining assertions are Proposition~\ref{prop:generic-dec}(b),(e), together with Proposition~\ref{prop:miao}(c).
\end{proof}

\section{Proof of Theorem~\ref{thm:main} and equality case remarks}\label{sec:main-proof}

\begin{proof}[Proof of Theorem~\ref{thm:main}]
Fix $\epsilon>0$.

\medskip\noindent\textit{Step 1: Strict DEC.} By Theorem~\ref{thm:strict-deform} applied with $\epsilon/2$, there is an initial data set with corners $(\bar g_\pm,\bar k_\pm)$ along $\Sigma$ satisfying the strict dominant energy condition in each region, whose Bartnik data satisfy the corner condition \eqref{eq:corner-condition}, with $|E(\bar g_+,\bar k_+)-E(g_+,k_+)|+|P(\bar g_+,\bar k_+)-P(g_+,k_+)|<\epsilon/2$, and with the interior data $C^{2,\alpha}$ up to $\Sigma$ (Proposition~\ref{prop:strict-interior}(e)). By Proposition~\ref{prop:strict-interior}(b) and Lemma~\ref{lem:apriori}, the interior strict gap satisfies $\bar\mu_--|\bar J_-|_{\bar g_-}\ge a_0>0$ on $\Omega$, and by Proposition~\ref{prop:exterior-matching}(b) the exterior gap satisfies $\bar\mu_+-|\bar J_+|_{\bar g_+}\ge c_0(1+r)^{-2-q_1}>0$ on $M\setminus\Omega$, for constants $a_0,c_0>0$ depending only on the fixed choice of $\epsilon$, independent of $\delta$.

\medskip\noindent\textit{Step 2: Mollification.} Apply Proposition~\ref{prop:miao} to $(\bar g_\pm,\bar k_\pm)$. Its hypotheses hold: Theorem~\ref{thm:strict-deform}(i)--(ii) gives the strict DEC in each region and the corner condition; the interior metric and tensor have class $C^{2,\alpha}\times C^{1,\alpha}$ up to $\Sigma$ by Proposition~\ref{prop:strict-interior}(e); and the exterior metric and tensor have class $C^{2,\alpha}\times C^{1,\alpha}$ by Proposition~\ref{prop:exterior-matching}. For each small $\delta>0$ this yields an initial data set $(g_\delta,k_\delta)$ of class $C^{2,\alpha}\times C^{1,\alpha}$ on $M$ agreeing with $(\bar g_\pm,\bar k_\pm)$ outside $U_\delta$, in particular retaining the strict DEC gaps of Step~1 there, with $\|(\mu_\delta-|J_\delta|_{g_\delta})_-\|_{L^{n/2}(M)}\le C\delta^{2/n}\to0$, and with $E(g_\delta,k_\delta)=E(\bar g_+,\bar k_+)$, $P(g_\delta,k_\delta)=P(\bar g_+,\bar k_+)$ exactly.

\medskip\noindent\textit{Step 3: Conformal correction.} By Corollary~\ref{cor:app-dec}, for $\delta$ small enough there is $(\widetilde g_\delta,\widetilde k_\delta)$ of class $C^{2,\alpha}\times C^{1,\alpha}$ on $M$, satisfying the dominant energy condition everywhere, with
\[
    E(\widetilde g_\delta,\widetilde k_\delta)=E(\bar g_+,\bar k_+)+O(\delta), \qquad
    P(\widetilde g_\delta,\widetilde k_\delta)=P(\bar g_+,\bar k_+).
\]
Fix $\delta$ small enough that the $O(\delta)$ term is below $\epsilon/2$.

\medskip\noindent\textit{Step 4: Applying the spacetime PMT.} The data $(\widetilde g_\delta,\widetilde k_\delta)$ are complete, asymptotically flat, and satisfy the dominant energy condition on $M$. For $3\le n\le7$, choose $\widehat q$ with $\frac{n-2}{2}<\widehat q<q$. The weighted H\"older decay of the corrected data then implies the weighted Sobolev hypotheses of \cite[Definition~3]{EHLS}, so \cite[Theorem~1]{EHLS} gives $E(\widetilde g_\delta,\widetilde k_\delta)\ge |P(\widetilde g_\delta,\widetilde k_\delta)|$ in these dimensions. For $n\ge8$, we further assume that $g_+-\delta\in C^{m,\alpha}_{-q}$ and $k_+\in C^{m-1,\alpha}_{-1-q}$ for every integer $m\ge2$, so that the harmonic asymptotic data arising from the higher-regularity EHLS reduction satisfy the asymptotic hypotheses of \cite[Theorem~1.1]{BrendleWang}. The reduction of $E\ge |P|$ to $E\ge0$ described in the final remark of Section~6 of \cite{EHLS}, followed by a further application of \cite[Theorem~18]{EHLS} and compactly supported smoothing preserving strict DEC, rules out $E<|P|$ by \cite[Theorem~1.1]{BrendleWang}; cf.~\cite[Theorem~4.4]{HKLZ}. Thus $E(\widetilde g_\delta,\widetilde k_\delta)\ge |P(\widetilde g_\delta,\widetilde k_\delta)|$ in every dimension $n\ge3$.

By Steps~1 and~3,
\[
    |E(\widetilde g_\delta,\widetilde k_\delta)-E(g_+,k_+)|<\epsilon,
    \qquad
    |P(\widetilde g_\delta,\widetilde k_\delta)-P(g_+,k_+)|<\epsilon.
\]
Hence
\[
    E(g_+,k_+) > E(\widetilde g_\delta,\widetilde k_\delta)-\epsilon
    \ge |P(\widetilde g_\delta,\widetilde k_\delta)|-\epsilon
    > |P(g_+,k_+)|-2\epsilon.
\]
Since $\epsilon>0$ was arbitrary, $E(g_+,k_+)\ge|P(g_+,k_+)|$.
\end{proof}

\medskip

\noindent\textbf{The equality case.} It is natural in this context to ask what equality $E=|P|$ forces in Theorem~\ref{thm:main}. In the smooth setting without corners, Huang--Lee proved, under their asymptotic hypotheses, that $E=|P|$ implies $E=|P|=0$ \cite{HL20}, and completed the rigidity by showing that such data arise from an isometric embedding into Minkowski space \cite{HL25}. They consider a \emph{lapse--shift pair} $(f,X)$, asymptotic to $(E,-2P)$, solving the adjoint linearized modified constraint system and produced by a variational argument that uses the positive mass inequality for all nearby initial data satisfying the dominant energy condition. Since Theorem~\ref{thm:main} provides an analogous inequality within the class of data with corners satisfying \eqref{eq:corner-condition}, one might expect the equality case with corners to produce lapse--shift pairs $(f_\pm,X_\pm)$ on $\Omega$ and $M\setminus\Omega$ with matching conditions along $\Sigma$ and, under analogous asymptotic hypotheses, that the data arise, in some fashion, from pieces of a spacelike hypersurface in Minkowski space. Carrying this out likely requires extending the Huang--Lee approach to the corner setting, which could be taken up in future work.

\appendix
\section{Conformal change formulas and weighted spaces}\label{sec:appendix}

\subsection{Conformal change formulas}\label{sec:conformal}
We set $c_n = \frac{n-2}{2(n-1)}$. For a conformal change $\bar{g} = u^{4/(n-2)}g$, $\bar{k} = u^{2/(n-2)}k$ with $u > 0$, one has
\begin{align}
        \bar{\mu} &= u^{-\frac{4}{n-2}}\Bigl(\mu - c_n^{-1}\,u^{-1}\Delta_g u\Bigr), \label{eq:conf-mu}\\
        \bar{J} &= u^{-\frac{6}{n-2}}\Bigl(J + c_n^{-1}\,u^{-1}k^\#(\nabla u)\Bigr).
\end{align}
Since $\bar{J}$ is a $(1,0)$-tensor, $|\bar{J}|_{\bar{g}} = u^{2/(n-2)}|\bar{J}|_g$, so
\begin{equation} \label{eq:conf-J}
        |\bar{J}|_{\bar{g}} \le u^{-\frac{4}{n-2}}\Bigl(|J|_g + c_n^{-1}\,u^{-1}|k|_g|\nabla u|_g\Bigr).
\end{equation}
Combining \eqref{eq:conf-mu} and \eqref{eq:conf-J}:
\begin{equation} \label{eq:conf-ineq}
        \bar{\mu} - |\bar{J}|_{\bar{g}} \ge u^{-\frac{4}{n-2}}\Bigl(\mu - |J|_g - c_n^{-1}\,u^{-1}\Delta_g u - c_n^{-1}\,u^{-1}|k|_g|\nabla u|_g\Bigr).
\end{equation}

\subsection{Boundary data under conformal change}\label{sec:conf-bdry}
If $u = 1$ on $\Sigma$, then $\bar\gamma = \gamma$, $\bar\omega = \omega$, and $\bar\tau = \tau$. However, the mean curvature changes:
\begin{equation} \label{eq:H-change}
        \bar{H} = H + c_n^{-1}\,\partial_\nu u,
\end{equation}
where $\nu$ is the chosen unit normal.

\subsection{Weighted spaces}\label{sec:weighted-spaces}

\begin{definition}
Let $B_R\subset\R^n$ be the closed Euclidean ball of radius $R$ centered at the origin, let $\Phi:M\setminus\mathcal K\to \R^n\setminus B_R$ be an asymptotic chart, and write $r=|x|$ on $\R^n\setminus B_R$. For $k\in\mathbb{N}_0$, $\alpha\in(0,1)$, and $q\in\R$, define
\[
\begin{split}
    \|f\|_{C^{k,\alpha}_{-q}(\R^n\setminus B_R)}
    :=\;&\sum_{|I|\le k}\sup_{x\in \R^n\setminus B_R}
        r(x)^{q+|I|}\,|\partial^I f(x)| \\
      &+\sum_{|I|=k}\sup_{\substack{x,y\in \R^n\setminus B_R\\0<|x-y|\le r(x)/2}}
        r(x)^{q+k+\alpha}
        \frac{|\partial^I f(x)-\partial^I f(y)|}{|x-y|^\alpha}.
\end{split}
\]
For $1\le p<\infty$, define
\[
    \|f\|_{W^{k,p}_{-q}(\R^n\setminus B_R)}
    :=\left(
       \int_{\R^n\setminus B_R}\sum_{|I|\le k}
       \bigl|r^{q+|I|}\partial^I f\bigr|^p\,r^{-n}\,dx
       \right)^{1/p}.
\]
On $M$, add the usual unweighted $C^{k,\alpha}$ and $W^{k,p}$ norms in a fixed finite precompact atlas covering $\mathcal K$. The spaces $C^{k,\alpha}_{-q}(M)$ and $W^{k,p}_{-q}(M)$ consist of functions with finite resulting norm, and $L^p_{-q}:=W^{0,p}_{-q}$. Tensor norms are taken componentwise.
\end{definition}


\end{document}